\documentclass[reqno]{amsart}
\usepackage{amssymb,amsmath,amsthm,amstext,amsfonts}
\usepackage[all]{xy}
\usepackage{graphicx,color}
\usepackage[ansinew]{inputenc}
\usepackage{hyperref}
\usepackage{here}
\usepackage{enumitem}
\usepackage[titletoc]{appendix}

\hypersetup{ colorlinks   = true, %Colours links instead of ugly boxes
urlcolor  = blue, %Colour for external hyperlinks
linkcolor    = blue, %Colour of internal links
citecolor   = blue %Colour of citations
}

\makeatletter \@addtoreset{equation}{section} \makeatother

\renewcommand\thetable{\thesection.\@arabic\c@table}

\theoremstyle{plain}
\newtheorem{maintheorem}{Theorem}

\newtheorem{maincorollary}{Corollary}

\newtheorem{theorem}{Theorem }[section]
\newtheorem{proposition}[theorem]{Proposition}
\newtheorem{lemma}[theorem]{Lemma}
\newtheorem{corollary}[theorem]{Corollary}
\newtheorem{claim}{Claim}
\theoremstyle{definition} \theoremstyle{remark}
\newtheorem{remark}[theorem]{Remark}

\newtheorem{definition}[theorem]{Definition}

\newtheorem{question}{Question}

\newtheorem{conj}{Conjecture}

\newcommand{\Z}{\mathcal{Z}_{\mathrm{Diff}}}
\newcommand{\N}{\mathbb{N}}
\newcommand{\R}{\mathbb{R}}

\newcommand{\information}{{
  \bigskip
  \footnotesize
  	\textbf{Davi Obata}: \textsc{CNRS-Laboratoire de Math\'ematiques d'Orsay, UMR 8628, Universit\'e Paris-Sud 11, Orsay Cedex 91405, France } \par\nopagebreak
  \textsc{Instituto de Matem\'atica, Universidade Federal do Rio de Janeiro, P.O. Box 68530, 21945-970, Rio de Janeiro Brazil}\par\nopagebreak
  \textit{E-mail:} \texttt{davi.obata@math.u-psud.fr}
}}

\begin{document}

\title{Symmetries of vector fields: the diffeomorphism centralizer}
\author{Davi Obata}
\thanks{D.O. was supported by the ERC project 692925 NUHGD}
\maketitle

\begin{abstract}
In this paper we study the diffeomorphism centralizer of a vector field: given a vector field it is the set of diffeomorphisms that commutes with the flow. Our main theorem states that for a $C^1$-generic diffeomorphism having at most finitely many sinks or sources, the diffeomorphism centralizer is quasi-trivial. In certain cases, we can promote the quasi-triviality to triviality. We also obtain a criterion for a diffeomorphism in the centralizer to be a reparametrization of the flow.
\end{abstract}
\section{Introduction}
Given a dynamical system, it is natural to try to understand its different types of symmetries. For diffeomorphisms, one set of these symmetries that has been studied it is the so-called centralizer. 

Let $M$ be a connected compact riemannian manifold. Given a diffeomorphism $g:M \to M$, the $C^r$-centralizer of $g$ is the set of $C^r$-diffeomorphisms $f$ that commutes with $g$. If the centralizer is generated by $g$, then it is called trivial. In \cite{Smale0} and \cite{Smale}, Smale asked the following question:
\begin{question}[\cite{Smale0}, \cite{Smale}]
\label{q.smale}
Is the set of $C^r$ diffeomorphisms with trivial centralizer a residual set? That is, does it contain a dense $G_{\delta}$-subset of the space of $C^r$ diffeomorphisms? 
\end{question} 

This question remains open in this generality. However, there are several partial answers, see for instance \cite{Kopell}, \cite{PalisYoccoz01}, \cite{PalisYoccoz02}, \cite{Burslem}, \cite{Fisher2008}, \cite{Fisher2009} and \cite{BakkerFisher} for some of these results. In \cite{BonattiCrovisierWilkinson}, Bonatti-Crovisier-Wilkinson proved that the centralizer of a $C^1$-generic diffeomorphism is trivial, giving a positive answer to Question \ref{q.smale} for $r=1$. 

For vector fields there are different types of symmetries that one can consider. Given a $C^r$-vector field $X$, one can look at the $C^r$-vector field centralizer of $X$, which is the set of $C^r$-vector fields $Y$ that commute with $X$. 

There are several works that study the different types of triviality of the vector field centralizer (see \cite{LeguilObataSantiago} for the different notions of triviality). In \cite{KatoMorimoto} the authors proved that an Anosov flow has quasi-trivial vector field centralizer. This result was extended to (Bowen-Walters) expansive flows by Oka in \cite{Oka}. In \cite{Sad}, Sad proved the triviality of the vector field centralizer for an open and dense subset of $C^{\infty}$-Axiom A vector fields that verify a strong transversality condition. In \cite{BonomoRochaVarandas}, Bonomo-Rocha-Varandas proved the triviality of the centralizer of transitive Komuro expansive flows, which includes the Lorenz attractor. 

Bonomo-Varandas proved in \cite{BonomoVarandas2} that a $C^1$-generic divergence free vector field has trivial vector field centralizer (they also obtain a generic result for Hamiltonian flows in the same paper). In a different paper, \cite{BonomoVarandas}, Bonomo-Varandas obtain that $C^1$-generic sectional axiom A vector fields have trivial vector field centralizer (see the introduction of \cite{BonomoVarandas} for the definition of sectional axiom A). 

Recently, the author with Leguil and Santiago in \cite{LeguilObataSantiago} have proved that $C^1$-generic vector fields have quasi-trivial vector field centralizer, ``extending'' the result of \cite{BonattiCrovisierWilkinson} for this type of centralizer. 

In this work, we will study a different type of centralizer. Given a $C^r$-vector field $X$, we denote by $X_t$ the flow on time $t$ generated by $X$. For any $1\leq s \leq r$, we define the \textit{$C^s$-diffeomorphism centralizer of $X$} as
\begin{equation}
\label{def-centralizer}
\Z^s(X) := \{f\in \mathrm{Diff}^s(M): f\circ X_t = X_t \circ f,\textrm{ } \forall t\in \R\}.
\end{equation}
This is the set of diffeomorphisms that commutes with the flow. Throughout this paper we will refer to the $C^1$-diffeomorphism centralizer of a vector field $X$ as the $C^1$-centralizer of $X$. We remark that this type of centralizer is less rigid than the vector field centralizer, see section \ref{sec.furtherremarks} for a discussion on that.

We define two types of ``triviality'' for this centralizer. Given a $C^r$-vector field  $X$, we say that it has \textit{quasi-trivial $C^s$-centralizer} if for any $f\in \mathcal{Z}^s_{\mathrm{Diff}}(X)$, $f$ is a reparametrization of the flow $X_t$, that is, there exists a continuous function $\tau: M \to \R$ such that $f(.) = X_{\tau(.)}(.)$. If $X$ has quasi-trivial $C^s$-centralizer and for every element of the centralizer $f$, the function $\tau(.)$ is constant, then we say that the $C^s$-centralizer is \textit{trivial}. 

Observe that the set $\{X_t(.): t\in \R\}$ is always contained in the centralizer of $X$. So a vector field has trivial centralizer if the centralizer is the smallest possible. In this paper, we are interested in the following version of Question \ref{q.smale} for flows:

\begin{question}
\label{q.2}
For a $C^1$-generic vector field $X$, is its $C^1$-centralizer quasi-trivial? Is it trivial?
\end{question}
In what follows, we refer the reader to section \ref{sec.preliminaries} for precise definitions of basic dynamical objects. 

It is natural to study generic systems that present some form of ``hyperbolicity''. In this paper we will focus on $C^1$-generic vector fields that have at most finitely many sinks or sources. In \cite{AbdenurBonattiCrovisier}, the authors proved that such systems have a weak form of hyperbolicity named dominated splitting (see theorem \ref{thm.abc}). Our main result is the following:

\begin{maintheorem}
\label{thm.maintheorem}
There exists a $C^1$-residual subset $\mathcal{R} \subset \mathcal{X}^1(M)$ such that if $X\in \mathcal{R}$ has at most finitely many sinks or sources, then $X$ has quasi-trivial $C^1$-centralizer. Moreover, if in addition $X$ has at most countably many chain-recurrent classes, then $X$ has trivial $C^1$-centralizer.  
\end{maintheorem}

We state some consequences of this theorem. In \cite{Peixoto2}, Peixoto proved that a $C^1$-generic vector field on a compact surface is Morse-Smale. Recall that a vector field is Morse-Smale if the non-wandering set is the union of finitely many hyperbolic periodic orbits and hyperbolic singularities, and it verifies some transversality condition. In particular, the non-wandering set is finite. As a consequence of Theorem \ref{thm.maintheorem} and the result of Peixoto, we have the following corollary. 
\begin{maincorollary}\label{cor.surfaces}
Let $M$ be a compact connected surface. Then, there exists a residual subset $ \mathcal{R}_\dagger \subset \mathfrak{X}^1(M)$ such that for any $X \in \mathcal{R}_\dagger$,  the $C^1$-centralizer of $X$ is trivial. 
\end{maincorollary}

 A $C^1$-vector field $X$ is \emph{Axiom A} if the non-wandering set is hyperbolic and the periodic points are dense in the non-wandering set. It is well known that Axiom A vector fields have finitely many chain-recurrent classes.

\begin{maincorollary}\label{cor.gentrivialityaxioma}
A $C^1$-generic Axiom A vector field has trivial $C^1$-centralizer.
\end{maincorollary}

\begin{remark}
Corollary \ref{cor.gentrivialityaxioma} actually holds for more a general type of hyperbolic system called sectional Axiom A.
\end{remark}

Another corollary is for $C^1$-vector fields far from homoclinic tangencies in dimension three (see \cite{CrovisierYang} for precise definitions). By the proof of the Palis conjecture in dimension three given in \cite{CrovisierYang},  a $C^1$-generic $X \in \mathfrak{X}^1(M)$ which cannot be approximated by such vector fields  is singular axiom A (or sectional axiom A), in particular, it has a finite number of chain-recurrent classes. Hence:

\begin{maincorollary}\label{cor.dim3}
Let $M$ be a compact connected $3$-manifold. Then there exists a residual subset $\mathcal{R}_\ddagger\subset \mathfrak{X}^1(M)$ such that any vector field $X \in \mathcal{R}_\ddagger$ which cannot be approximated by vector fields exhibiting a homoclinic tangency has trivial $C^1$-centralizer.  
\end{maincorollary}

To prove Theorem \ref{thm.maintheorem}, we will need the proposition below. This proposition deals with the construction of the reparametrization of the flow, given a diffeomorphism in the centralizer that fixes orbits. We remark that this type of construction does not appear for diffeomorphisms.

\begin{proposition}
\label{prop.mainpropextension}
Let $X\in \mathfrak{X}^1(M)$ be a $C^1$-vector field whose periodic orbits and singularities are all hyperbolic. Let $f\in \Z^1(X)$ be an element of the centralizer with the following property: there exists a constant $T>0$ such that for every $p\in M$, we have $f(p) \in X_{[-T,T]}(p)$, where $X_{[-T,T]}(p)$ is the piece of orbit of $p$ from time $-T$ to $T$. Then there exists an $X$-invariant continuous function $\tau:M\to \R$ such that $f(.) = X_{\tau(.)}(.)$.
\end{proposition}

Observe that in Theorem \ref{thm.maintheorem}, without the additional assumption of at most countably many chain recurrent classes, we do not get the triviality of the $C^1$-centralizer for a $C^1$-generic vector field that has at most finitely many sinks or sources.  What is missing to obtain the triviality of the $C^1$-centralizer in this case it is to prove that for a $C^1$-generic vector field every invariant continuous function is constant. This was conjectured (without precision on the regularities) by Ren\'e Thom (\cite{Thom}). 

\begin{conj}[\cite{Thom}]
 For a $C^1$-generic vector field, any $C^1$ (or $C^0$) invariant function of the manifold is constant.
\end{conj}

Also, after the work \cite{LeguilObataSantiago}, to conclude that a $C^1$-generic vector field has trivial $C^1$-vector field centralizer it is equivalent to proving Thom's conjecture.

Our approach to prove Theorem \ref{thm.maintheorem} is an adaptation for flows of the approach used by Bonatti-Crovisier-Wilkinson in \cite{BonattiCrovisierWilkinson}. We organize this paper as follows. In section \ref{sec.preliminaries} we will introduce some basic notions and notations of vector fields that we will use, we will also recall the main tools from $C^1$-generic dynamics that will be used. The proof of proposition \ref{prop.mainpropextension} is given in section \ref{sec.proofprop}. The proof of Theorem \ref{thm.maintheorem} is given in section \ref{sec.prooftheorem}. We conclude this paper with one example that justifies our claim that the diffeomorphism centralizer is less rigid than the vector field centralizer, in section \ref{sec.furtherremarks}.

\subsection*{Acknowledgments:}
The author would like to thank Sylvain Crovisier for all his patience and guidance with this project. The author also benefited from conversations with Alexander Arbieto, Martin Leguil, Mauricio Poletti and Bruno Santiago. 

\section{Preliminaries}
\label{sec.preliminaries}
In this section we introduce the notations we will use throughout this article and state some preliminary results on $C^1$-generic dynamics that will be used in our proofs.
 
\subsection{General notions on vector fields}

Let  $M$ be a  smooth manifold of dimension $d \geq 1$,  which we assume to be compact and boundaryless. For any $r \geq 1$, we denote by $\mathfrak{X}^r(M)$ the space of vector fields over $M$, endowed with the $C^r$ topology. A property $\mathcal{P}$ for vector fields in $\mathfrak{X}^r(M)$ is called $C^r$-generic if it is satisfied for any vector field in a  \textit{residual set} of $\mathfrak{X}^r(M)$. Recall that $\mathcal{R} \subset \mathfrak{X}^r(M)$ is \textit{residual} if it contains a dense $G_{\delta}$-subset of $\mathfrak{X}^r(M)$. 

In the following, given a vector field $X\in\mathfrak{X}^1(M)$, we denote by $\mathrm{Sing}(X):=\{x \in M: X(x)=0\}$ the set of \textit{singularities} (or \textit{zeros}) of $X$. The set of (non-singular) periodic points will be denoted by $\mathrm{Per}(X)$, and we set $\mathrm{Crit}(X) = \mathrm{Per}(X) \cup \mathrm{Sing}(X)$.

For any $x \in M$ and any interval $I \subset \R$, we also let $X_I(x):=\{X_t(x):t \in I\}$. In particular, we denote by $orb(x):=X_\R(x)$  the orbit of the point $x$ under $X$.

Let $X\in\mathfrak{X}^1(M)$ be some $C^1$ vector field. The \textit{non-wandering set} $\Omega(X)$ of $X$ is defined as the set of all points $x \in M$ such that for any open neighbourhood $U$ of $x$ and for any $T >0$, there exists a time $t >T$ such that $U\cap X_t(U) \neq \emptyset$.

Let us also recall another weaker notion of recurrence. Given two points $x,y \in M$, we write $x\prec_X y$ if for any $\varepsilon>0$ and $T>0$, there exists an $(\varepsilon,T)$-\textit{pseudo orbit} connecting them, i.e., there exist $n \geq 2$, $ t_1,t_2,\dots,t_{n-1} \in [T,+\infty)$,  and $x=x_1,x_2,\dots,x_{n}=y\in M$, such that $d(X_{t_j}(x_j),x_{j+1})< \varepsilon$, for $j \in \{1,\dots,n-1\}$.  The \textit{chain recurrent set} $\mathcal{CR}(X)\subset M$ of $X$ is defined as the set of all points $x \in M$ such that $x \prec_X x$. Restricted to $\mathcal{CR}(X)$, we consider the equivalence relation given by $x\sim_X y$ if and only if $x\prec_X y$ and $y\prec_X x$. An equivalence class under the relation $\sim_X$ is called a \textit{chain recurrent class}: $x, y \in \mathcal{CR}(X)$ belong to the same chain recurrent class if $x \sim_X y$. In particular, chain recurrent classes define a partition of the chain recurrent set $\mathcal{CR}(X)$. 
 
An $X$-invariant compact set $\Lambda$ is \textit{hyperbolic} if there is a continuous decomposition of the tangent bundle over $\Lambda$, $T_{\Lambda}M = E^s \oplus \langle X \rangle \oplus E^u$ into $DX_t$-invariant subbundles that verifies the following property: there exists $T>0$ such that for any $x\in \Lambda$, we have
\[
\|DX_T(x)|_{E^s_x}\| < \frac{1}{2} \textrm{ and } \|DX_{-T}(x)|_{E^u_x}\|< \frac{1}{2}.
\]   
A periodic point $x\in \mathrm{Per}(X)$ is hyperbolic if $orb(x)$ is a hyperbolic set. Let $p\in \mathrm{Per}(X)$ be a hyperbolic periodic point and $\gamma_p$ be its orbit. We define its strong stable manifold as
\begin{equation}
\label{eq.defstrongstable}
W^{ss}(p) :=\{ x\in M: d(X_t(x), X_t(p)) \xrightarrow{t\to +\infty} 0\}.
\end{equation}
The stable manifold theorem states that $W^{ss}(p)$ is an immersed submanifold of dimension $\mathrm{dim}(E^s(p))$ tangent to $E^s(p)$ at $p$. We define the \textit{stable manifold} of the orbit of $p$ as
\[
W^s(p) = \displaystyle \bigcup_{q\in \gamma_p} W^{ss}(q).
\]  
A hyperbolic periodic orbit is a sink if the unstable direction is trivial. It is a source if the stable direction is trivial. A hyperbolic periodic orbit is a saddle if it is neither a sink nor a source.

\subsection{$C^1$-generic dynamics, the unbounded normal distortion and large normal derivative properties}

In this part we will present the main tools from $C^1$-generic dynamics that we will need to prove Theorem \ref{thm.maintheorem}. Let us first fix some notation.
 
For a given vector field $X\in \mathfrak{X}^1(M)$, we define the non-singular set as $M_X : = M-Sing(X)$. Observe that for a fixed riemannian metric, for any point $p\in M_X$, it is well defined the subspace orthogonal to the vector field direction, $N^X(p) = \langle X(p) \rangle^{\perp}$. This define the normal bundle 
\[
N^X = \displaystyle \bigsqcup_{p\in M_X} N^X(p)
\]
over $M_X$. Let $\Pi_X: TM_X \to N^X$ be the orthogonal projection on $N^X$. Whenever it is clear that we have fixed a vector field $X$, we will denote the normal bundle and othogonal projection by $N$ and $\Pi$. 

On $N^X$ we have a well defined flow, called the \textit{linear Poincar\'e flow} defined as follows: for any point $p\in M_X$, vector $v\in N^X(p)$ and time $t\in \R$, the image of $v$ by the linear Poincar\'e flow is
\[
P_{t}(p).v = \Pi_X(X_t(p)) \circ DX_t(p)v.
\] 
Next, we define two notions that will be crucial in our proof. 
\begin{definition}[Unbounded normal distortion (UND)]
\label{def.und}
Let $X\in \mathfrak{X}^1(M)$ be a $C^1$-vector field. We say that $X$ verifies the \textit{unbounded normal distortion} property if the following holds: there exists a dense subset $\mathcal{D} \subset M- \mathcal{CR}(X)$, such that for any $K \geq 1$, $x\in \mathcal{D}$ and $y\in M-\Omega(X)$ verifying $y\notin orb(x)$, there is $n\in (0,+\infty)$, such that
\[
|\log \det P_{n}(x) - \log \det P_{n}(y)| > K.
\] 
\end{definition}

\begin{definition}[Unbounded normal distortion on stable manifolds ($\mathrm{UND}^s$)]
\label{def.unds}
Let $X\in \mathfrak{X}^1(M)$ be a $C^1$-vector field. We say that $X$ verifies the \emph{unbounded normal distortion on stable manifolds} property if the following holds: for any $p\in Crit(X)$, there exists a dense subset $\mathcal{D}^s_p \subset W^s(p)$, such that for any $K \geq 1$, $x\in \mathcal{D}^s_p$ and $y\in W^s(p)$ verifying $y\notin orb(x)$,  there is $n\in (0,+\infty)$, such that
\[
|\log \det \left(P_{n}(x)|_{T_{x}W^s(p)}\right) - \log \det \left(P_{n}(y)|_{T_yW^s(p)}\right)| > K.
\] 
\end{definition}
Given a vector field $X$ that has the $\mathrm{UND}^s$ property, we define
\[
\mathcal{D}^s = \displaystyle \bigcup_{p\in \mathrm{Per}(X)} D^s_p.
\]

In \cite{BonattiCrovisierWilkinson}, the authors introduce these notions for diffeomorphisms and they use it as the main ingredient to obtain triviality of the centralizer in an open and dense subset of the manifold. They prove that these properties actually hold $C^1$-generically. For vector fields, the $C^1$-genericity of the UND property was proved in \cite{LeguilObataSantiago}, and the $C^1$-genericity of the $\mathrm{UND}^s$ property was proved in \cite{BonomoVarandas}. We summarize it in the following theorem:

\begin{theorem}[\cite{LeguilObataSantiago} and \cite{BonomoVarandas}]
\label{thm.undundsgeneric}
There exists a residual subset $\mathcal{R}_1\subset \mathfrak{X}^1(M)$ such that any vector field $X\in \mathcal{R}^1$ verifies the UND and $\mathrm{UND}^s$ properties. 
\end{theorem} 

Given a non-singular invariant set $\Lambda\subset M_X$, we say that it admits a dominated splitting for the linear Poincar\'e flow if there exists a $P_t$-invariant, non-trivial decomposition of the normal bundle $N_{\Lambda}^X = E \oplus F$ and a constant $T>0$ such that for any $x\in \Lambda$ 
\[
\|P_{T}(x)|_E\|.\|(P_{T}(x)|_{F})^{-1}\| <\frac{1}{2}.
\]

In \cite{AbdenurBonattiCrovisier}, the authors proved that $C^1$-generically for a diffeomorphism far from the existence of infinitely many periodic sinks or sources (Newhouse phenomenon), one can obtain that the non-wandering set is decomposed into the disjoint union of finitely many periodic sinks or sources and invariant sets each of which admits a dominated decomposition. The key ingredient in their proof is a result of Bonatti-Gourmelon-Vivier (see corollary 2.19 in \cite{BonattiGourmelonVivier}), which is a generalization of a previous theorem in \cite{BonattiDiazPujals}. In \cite{BonattiGourmelonVivier}, the authors also prove a version of corollary 2.19 for flows, given by corollary 2.22. Using this, it is easy to adapt the proof of Abdenur-Bonatti-Crovisier to obtain the following statement:

\begin{theorem}[\cite{AbdenurBonattiCrovisier}]
\label{thm.abc}
There exists a residual subset $\mathcal{R}_2 \subset \mathfrak{X}^1(M)$ such that for any $X\in \mathcal{R}_2$, either (1) or (2) holds:
\begin{enumerate}
\item the non-wandering set admits a decomposition
\begin{equation}
\label{eq.decomposition}
\Omega(X) = \mathrm{Sink}(X) \sqcup \mathrm{Source}(X) \sqcup \Lambda_1 \sqcup \cdots \sqcup \Lambda_{k_X},
\end{equation}
such that $\mathrm{Sink}(X)$ is the set of periodic sinks of $X$, the set $\mathrm{Source}(X)$ is the set of periodic sources of $X$, and each $\Lambda_i-Sing(X)$ admits a dominated splitting for the linear Poincar\'e flow;
\item there are infinitely many periodic sinks or sources.
\end{enumerate}
\end{theorem}

In \cite{BonattiCrovisierWilkinson}, the authors also introduce the notion of large derivative for diffeomorphisms (see section 2.3 in \cite{BonattiCrovisierWilkinson}). This is the key property to pass from triviality of the centralizer in an open and dense subset of $M$ to triviality in the entire manifold. For vector fields we introduce the following similar definition:

\begin{definition}[Large normal derivative (LND)]
\label{defi.lnd}
A vector field $X\in \mathfrak{X}^1(M)$ satisfies the LND property if for any $K>0$, there exists $T = T(K)>0$ such that for any $p\in M_X$ and $t> T$, there exists $s\in \R$ that verifies:
\[
\max \{\|P_t(X_s(p))\|, \|P_{-t}(X_{s+T}(p))\|\}>K.
\]
\end{definition}

If the chain-recurrent set admits a decomposition as in (\ref{eq.decomposition}), then the LND property holds. This was remarked for diffeomorphisms by Bonatti-Crovisier-Wilkinson (see remark 8 of Appendix A in \cite{BonattiCrovisierWilkinson}). The same holds in our context for flows, we make it precise in the following corollary:

\begin{corollary}
\label{cor.LNDfarsinks}
Let $\mathcal{R}_2\subset \mathfrak{X}^1(M)$ be the residual subset from theorem \ref{thm.abc}. If $X\in \mathcal{R}_2$ and $X$ does not have infinitely many sinks or sources, then $X$ has the LND property.
\end{corollary}

In the next statement, we summarize some other $C^1$-generic properties that we will use.
\begin{theorem}
\label{thm.residualkupkabc}
There exists $\mathcal{R}_3\subset \mathfrak{X}^1(M)$ a residual subset such that any $X\in \mathcal{R}_3$ verifies:
\begin{itemize}
\item every periodic orbit, and singularity, is hyperbolic (Kupka-Smale);
\item two distinct periodic orbits have different periods;
\item any connected component $O$ of the interior of $\Omega(X)$ is contained in the closure of the stable manifold of a periodic point (Bonatti-Crovisier, \cite{BonattiCrovisier}).
\end{itemize}
\end{theorem}

\section{Proof of Proposition \ref{prop.mainpropextension}}
\label{sec.proofprop}
Throughout this section we fix a vector field $X\in \mathfrak{X}^1(M)$ whose periodic orbits and singularities are all hyperbolic. We also fix $f\in \Z^1(X)$ a $C^1$-diffeomorphism in the centralizer of $X$ that verifies the conditions of proposition \ref{prop.mainpropextension}. The goal of this section is to construct an $X$-invariant continuous function $\tau:M \to \R$ such that $f(.) = X_{\tau(.)}(.)$. 

\subsection{Non-critical points}
In this subsection we construct the function $\tau$ for non-critical points. This is given in the following lemma:
\begin{lemma}
\label{lemma.functionnoncriticalpoints}
There exists an $X$-invariant continuous function $\tau_1:M-Crit(X) \to \R$ such that $f|_{M-Crit(X)}(.) = X_{\tau_1(.)}(.)$. 
\end{lemma}
\begin{proof}
Let $p\in M-Crit(X)$. Since $p$ is a non-critical point and $f$ fixes its orbit, there is an unique $T_p\in \R$ such that $f(p) = X_{T_p}(p)$. We claim that for any $q\in orb(p)$ we have $f(q) = X_{T_p}(q)$. 

Indeed, let $q\in orb(p)$ and let $s\in \R$ be such that $q=X_s(p)$. Hence 
\[
f(q) = f(X_s(p)) = X_s(f(p)) = X_s(X_{T_p}(p)) = X_{T_p}(X_s(p)) = X_{T_p}(q).
\]

Define $\tau_1(p) = T_p$ for any $p\in M-Crit(X)$.
\begin{claim}
\label{claim.1}
The function $\tau_1$ is continuous on $M-Crit(X)$.
\end{claim} 
\begin{proof}
Let $T>0$ be the constant that appears in the hypothesis of proposition \ref{prop.mainpropextension}, that is, for any $q\in M$, $f(q) \in X_{[-T,T]}(q)$. Fix $p\in M-Crit(X)$, we will prove that $\tau_1$ is continuous on $p$.

For each $\delta>0$, define $\mathcal{N}(p,\delta) := \exp_p(N(p,\delta))$, where $\exp_p$ is the exponential map on $p$ and $N(p,\delta)$ is the ball of radius $\delta$ inside $N(p) \subset T_pM$. For $\delta$ small enough, the following map is a $C^1$-diffeomorphism
\[
\begin{array}{rcl}
\Psi: (-T-1,T+1) \times \mathcal{N}(p,\delta) & \longrightarrow & M\\
(t,x) & \mapsto & X_t(x).
\end{array}
\]
Let $V= Im(\Psi)$. The pair $(\Psi^{-1}, V)$ is a flow box around the piece of orbit $X_{(-T-1,T+1)}(p)$. Let $(p_n)_{n\in \N}$ be a sequence of points contained in $\mathcal{N}(p,\delta) \cap M-Crit(X)$, which converges to $p$. Since $f(p_n) \in X_{[-T,T]}(p_n)$, we have that $\Psi^{-1}(f(p_n)) = (\tau_1(p_n),p_n)$. By the continuity of $f$, we obtain that $\displaystyle \lim_{n\to + \infty} \tau_1(p_n) = \tau_1(p)$ and $\tau_1$ is continuous on $M-Crit(X)$.
\end{proof}
This claim concludes the proof of lemma \ref{lemma.functionnoncriticalpoints}.
\end{proof}

\subsection{Periodic points}
In this subsection we prove that there exists a continuous extension of the function $\tau_1$ constructed in the previous section to the periodic points. We prove the following lemma:
\begin{lemma}
\label{lemma.functionperiodicpoints}
There exists an $X$-invariant continuous function $\tau_2: M-Sing(X) \to \R$ that verifies the following:
\begin{itemize}
\item $f|_{M-Sing(X)}(.) = X_{\tau_2(.)}(.)$;
\item $\tau_2|_{M-Crit(X)} = \tau_1$.
\end{itemize}
\end{lemma}
\begin{proof}
Let $p\in Per(X)$ be a hyperbolic periodic point and let $\pi(p)$ be the period of $p$. Notice that the equation $X_s(p) = f(p)$ has infinitely many solutions for $s\in \R$, thus the same strategy that we used for non-periodic points does not apply in this case. However, we will prove that the function $\tau_1$ is constant on $W^s(p)$ and $W^u(p)$. This will allow us to extend it to the periodic orbit. 

\begin{claim}
\label{claim.constantalongstablemanifoldperiodic}
The function $\tau_1$ is constant on $W^s(p)-orb(p)$ and $W^u(p) - orb(p)$. 
\end{claim} 
\begin{proof}
Recall that 
\[
W^s(p) = \displaystyle \bigcup_{t\in [0,\pi(p)]} W^{ss}(X_t(p))\]
and that the strong stable manifolds forms a foliation of the stable manifold, in particular, if $t,s \in [0, \pi(p))$ are such that $t\neq s$ then $W^{ss}(X_t(p)) \cap W^{ss}(X_s(p)) = \emptyset$.

 Observe that $W^s(p)-orb(p) \subset M-Crit(X)$, hence, the function $\tau_1$ is well defined on $W^s(p)- orb(p)$. Since the function $\tau_1$ is also $X$-invariant, it is enough to prove that $\tau_1$ is constant along $W^{ss}(p)-\{p\}$.

Since $f$ is a $C^1$-diffeomorphism that commutes with the flow, it is well known that $f(W^{ss}(p)) = W^{ss}(f(p))$. If $\tau_1$ was not constant along $W^{ss}(p)$ there would be two points $x,y \in W^{ss}(p)$ such that $0<|\tau_1(x)- \tau_1(y)|<\pi(p)$. Hence,
\[
f(x) = X_{\tau_1(x)}(x) \in W^{ss}(X_{\tau_1(x)}(p)) \textrm{ and } f(y) = X_{\tau_1(y)}(y) \in W^{ss}(X_{\tau_1(y)}(p)).
\]
However, $W^{ss}(X_{\tau_1(x)}(x)) \cap W^{ss}(X_{\tau_1(y)}(y))= \emptyset$. This implies that $f(W^{ss}(p))$ is not contained in any strong stable manifold, which is a contradiction with the fact that  $f(W^{ss}(p)) = W^{ss}(f(p))$. This proves that $f$ is constant on each connected component of $W^s(p) - orb(p)$. Observe that $W^s(p) - orb(p)$ has at most two connected components.

Since $f$ fixes the orbit of $p$, it induces a $C^1$-diffeomorphism $f^s$ on $W^s(p)$. If $\tau_1$ was not constant on $W^s(p)$, it would take two different values $T_1$ on the connected component $O_1$ and $T_2$ on the connected component $O_2$. From the above calculation, there exists $k\in \mathbb{Z}-\{0\}$ such that $T_2 = T_1 + k \pi(p)$. 

On $TW^s(p)$ we consider the normal bundle $N_s = \langle X \rangle^{\perp}$, for the riemannian metric induced by the metric of the manifold on $W^s(p)$. Let $\Pi^s: TW^s(p) \to N_s$ be the orthogonal projection on $N_s$ and let $P^s_t(.)$ be the linear Poincar\'e flow restricted to $W^s(p)$.

For any $q_1 \in O_1$ and $q_2 \in O_2$, we have the following formulas:
\[
\Pi^s(f^s(q_1))Df^s(q_1)\Pi^s(q_1) = P^s_{T_1}(q_1) \textrm{ and } \Pi^s(f^s(q_2))Df^s(q_2)\Pi^s(q_2) = P^s_{T_2}(q_2).
\]

Take $q_1^n$ a sequence in $O_1$ converging to $p$ and $q^n_2$ a sequence in $O_2$ converging to $p$. Since $f^s$ is $C^1$, we would have that
\[
\begin{array}{rcl}
\displaystyle \lim_{n\to + \infty} \Pi^s(f^s(q^n_1))Df^s(q^n_1)\Pi^s(q^n_1) & = &\\
\displaystyle \lim_{n\to + \infty} \Pi^s(f^s(q^n_2))Df^s(q^n_2)\Pi^s(q^n_2) &= &\Pi^s(f^s(p))Df^s(p)\Pi^s(p).
\end{array}
\]
However, since $p$ is a hyperbolic periodic point and $T_2 \neq T_1$, we have
\[
\|P^s_{T_1}(p)\| \neq \|P^s_{T_2}(p)\|.
\]
This is a contradiction and $\tau_1$ is constant on $W^s(p)$.
\end{proof}

Claim \ref{claim.constantalongstablemanifoldperiodic} implies that if $p$ is a periodic sink or source, then the function $\tau_2$ has a continuous extension to the orbit of $p$.

Let us assume that $p$ is a hyperbolic saddle. We remark that the following claim is independent of claim \ref{claim.constantalongstablemanifoldperiodic}.
\begin{claim}
\label{claim.saddle}
There is a constant $c_p\in \R$ such that $\tau_1|_{(W^s(p)-orb(p))} = \tau_1|_{(W^u(p)-orb(p))} = c_p$.
\end{claim}
\begin{proof}
The proof of this fact is the same as the proof of Proposition $4.4$ in \cite{LeguilObataSantiago}. For the sake of completeness we will repeat it here. 

Fix a point $p_s \in W^s(p) - orb(p)$. We will prove that for any point $p_u \in W^u(p)-orb(p)$, we have $\tau_1(p_u) = \tau_1(p_s)$. By the $X$-invariance of $\tau_1$, it is enough to consider $p_u \in W^u_{\mathrm{loc}}(\sigma)$. Let $(D^s_n)_{n\in \N}$ be a sequence of discs transverse to $W^u_{loc}(p)$ and with radius $\frac{1}{n}$. By the lambda-lemma (see \cite{PalisdeMelo} chapter 2.7), for each $n\in \N$ and for any disc $D^u$ transverse to $W^s_{loc}(p)$, there exists $t_n>0$ such that $X_{t_n}(D^u) \pitchfork D_n^s \neq \emptyset$. Since this holds for any disc $D^u$ and there are only countably many periodic orbits, for each $n\in \N$ we can find a disc $D^u_n$ centered in $p_s$ with radius smaller than $\frac{1}{n}$ and a point $q_n \in (X_{t_n}(D^u) \pitchfork D_n^s)$ which is non-periodic.

It is immediate that $q_n \to p_s$, as $n\to + \infty$. Since the function $\tau_1$ is continuous on $M_X- \mathrm{Crit}(X)$, we have that $\tau_1(q_n) \to \tau_1(p_s)$. We also have that $X_{t_n}(q_n) \to p_u$ as $n\to +\infty$. Hence, $\tau_1(X_{t_n}(q_n)) \to \tau_1(p_u)$. By the $X$-invariance of $\tau_1$,  we obtain
$$
\tau_1(p_s) = \displaystyle \lim_{n\to + \infty} \tau_1(q_n) = \lim_{n\to +\infty} \tau_1(X_{t_n}(q_n)) = \tau_1(p_u).
$$
This implies that for any $p_u \in W^u(p)-orb(p)$, we have $\tau_1(p_s) = \tau_1(p_u)$. Analogously, we can prove that for a fixed $p_u'\in W^u_{\mathrm{loc}}(p)-orb(p)$ and for any $p_s'\in W^s_{\mathrm{loc}}(p)-orb(p)$, it is verified $\tau_1(p_s') = \tau_1(p_u')$. We conclude that $\tau_1|_{W^s(p) - orb(p)} = \tau_1|_{W^u(p)- orb(p)} = c_p$, for some constant $c_p\in \R$. 
\end{proof}
From this claim, we can define an extension of $\tau_1$ to the set of periodic points by setting $\tau_2|_{orb(p)}:= c_p$, for $p\in Per(X)$. Let us prove that $\tau_2$ is continuous on $M-Sing(X)$.

Fix $p \in Per(X)$. Since $\tau_2$ is constant on $W^s(p)$, in the case that $p$ is a sink, it is immediate that $\tau_2$ is continuous on $p$. Similarly, we conclude continuity of $\tau_2$ on $p$ in the case that $p$ is a source. Suppose that $p$ is a saddle and let $(p_n)_{n\in \N}\subset M-Sing(X)$ be a sequence converging to $p$. 

Suppose first that the sequence $(p_n)_{n\in \N}$ is formed by non-periodic points, hence, by the continuity of $\tau_1$ and using the same argument as in the proof of claim \ref{claim.saddle}, we can conclude that $\displaystyle \lim_{n\to +\infty} \tau_2(p_n) = \tau_2(p)$.

In the case that $(p_n)_{n\in \N}$ is formed by periodic points, we may choose a sequence of points $(q_n)_{n\in \N} \subset M- Crit(X)$ such that each $n\in \N$, the point $q_n$ is contained in the stable manifold of $p_n$ and $d(q_n,p_n) < \frac{1}{n}$. Observe that $\lim_{n\to + \infty} q_n = p$. By claim \ref{claim.saddle}, we have that $\tau_2(q_n) = \tau_2(p_n)$, hence $\displaystyle \lim_{n\to + \infty} \tau_2(p_n) = \lim_{n\to + \infty} \tau_2(q_n) = \tau_2(p)$. We conclude that $\tau_2$ is continuous. \qedhere
\end{proof}
\subsection{Singularities}
We conclude the proof of proposition \ref{prop.mainpropextension}, by extending continuously the function $\tau_2$ from lemma \ref{lemma.functionperiodicpoints} to the singularities. This extension will give us the function $\tau$ that we wanted. We separate the proof into the case when the singularity is a saddle and when the singularity is a sink or source. This is given by the two lemmas below.

\begin{lemma}
Let $\sigma \in Sing(X)$ be a hyperbolic singularity which is a saddle. Then the function $\tau_2$ can be extended continuously to $\sigma$. 
\end{lemma}
\begin{proof}
The proof of this lemma is the same as the proof of claim \ref{claim.saddle} (see also Proposition $4.4$ in \cite{LeguilObataSantiago}).
\end{proof}

\begin{lemma}
\label{lem.sinksing}
Let $\sigma\in Sing(X)$ be a singularity which is a sink. Then the function $\tau_2$ is constant on $W^s(\sigma)$, in particular, it can be extended continuously to $\sigma$. A similar statement holds if $\sigma$ is a source.
\end{lemma}
\begin{proof}

Suppose that $\tau_2$ is not constant on $W^s(\sigma)$, then there exists an open set $U \subset W^s(\sigma)$ such that $\tau_2 (U) = (a,b)$, where $a\neq b$. Suppose also that $(a,b) \subset (0, + \infty)$. For each $t \in (a,b)$ we fix a point $x_t \in U$ such that $\tau_2(x_t) = t$. Thus, $f(x_t) = X_t(x_t)$. By the $X$-invariance of the function $\tau_2$, we obtain that $f^n(x_t) = X_{nt}(x_t)$. 

Let $\lambda_1, \cdots, \lambda_l \in \mathbb{C}$ be the eigenvalues of $DX(\sigma)$. For each $j\in \{1, \cdots, l\}$, consider the number $c_j = \mathfrak{R}(\lambda_j)$, where $\mathfrak{R}(.)$ is the real part of a number. Since $\sigma$ is a hyperbolic sink for $X$, we have that $c_j<0$, for each $j=1, \cdots, l$. For each $t\in (a,b)$ the value 
\begin{equation}
\label{eq.defiht}
h_t := \displaystyle \limsup_{n\to +\infty} \frac{1}{n}\log d(\sigma, f^n(x_t)) = \limsup_{n\to +\infty} \frac{1}{n} \log d(\sigma, X_{nt}(x_t))
\end{equation}
belongs to the set $C_t:= \{ tc_1 , \cdots,tc_l\}$.

Since $\sigma$ is a fixed point of the $C^1$-diffeomorphisms $f$, we have that $Df(\sigma)$ has at most $d$ different eigenvalues $\tilde{\lambda}_1, \cdots, \tilde{\lambda}_{k} \in \mathbb{C}$, where $1\leq k \leq d$. For each $i\in \{1, \cdots, k\}$, let $a_i = \log|\tilde{\lambda}_i|$ and let $A=\{a_1, \cdots, a_k\}$. We remark that if $q\in M$ is a point such that $\displaystyle \lim_{n \to +\infty} f^n(q) = p$, then
\begin{equation}
\label{eq.speedf}
\displaystyle \limsup_{n\to +\infty} \frac{1}{n}\log d(f^n(q), p ) \in A.
\end{equation}

Notice that the sets $C_t$ varies continuously with $t\in (a,b)$. Observe also that the set $(a,b)$ is uncountable and the set $A$ is finite. Therefore, there exists $t_0\in (a,b)$ such that $C_{t_0} \cap A = \emptyset$. By (\ref{eq.defiht}) and (\ref{eq.speedf}) this is a contradiction, since
\[
h_{t_0} \in A \textrm{ and } h_{t_0} \in C_{t_0}.
\]
We conclude that $\tau_2$ is constant on $W^s(\sigma)$. If $(a,b) \subset (-\infty, 0)$ we can repeat this argument for $f^{-1}$ and we would obtain the same conclusion. This  implies that the function $\tau_2$ can be continuously extended to a function $\tau$ defined on $\sigma$.
\end{proof}

\section{Proof of Theorem \ref{thm.maintheorem}}
\label{sec.prooftheorem}
In this section we prove Theorem \ref{thm.maintheorem}. Let $\mathcal{R}_1$, $\mathcal{R}_2$ and $\mathcal{R}_3$ be the residual subsets given by theorems \ref{thm.undundsgeneric}, \ref{thm.abc} and \ref{thm.residualkupkabc}, respectively. Consider the residual set $\mathcal{R} = \mathcal{R}_1 \cap \mathcal{R}_2 \cap \mathcal{R}_3$. We claim that $\mathcal{R}$ is the residual subset that verifies the conditions of Theorem \ref{thm.maintheorem}.

From now on we fix $X\in \mathcal{R}$ such that $X$ does not have infinitely many sinks or sources, and we fix $f\in \Z^1(X)$. Since $X$ is fixed, we will denote $\Pi_X$, and $N^X$, by $\Pi$, and $N$. 

We want to prove that $f$ verifies the conditions of \ref{prop.mainpropextension}. Let $\mathcal{D}$ and $\mathcal{D}^s$ be the subsets given by the UND and $\mathrm{UND}^s$ properties. 

\begin{lemma}
\label{lem.fixesDDs}
For each $x\in \mathcal{D} \cup \mathcal{D}^s$, we have that $f(orb(x)) = orb(x)$. 
\end{lemma}
\begin{proof}
Fix $x\in \mathcal{D}$. Since $f$ commutes with the flow $X_t$, it is easy to see that $M-\Omega(X)$ is an $f$-invariant set.

Notice that $DX_t(f(x)) = Df(X_t(x)) DX_t(x) Df^{-1}(f(x))$. Since the vector field direction, $\langle X \rangle$, is invariant by $Df$, $Df^{-1}$ and $DX_t$, we have that
\[\arraycolsep=1.2pt\def\arraystretch{1.6}
\begin{array}{l}
|\det P_t(f(x))| = |\det \Pi(X_t(f(x))) DX_t(f(x)) |_{N(f(x))}| =\\
|\det\Pi(X_t(f(x))) Df(X_t(x)) DX_t(x) Df^{-1}(f(x))|_{N(f(x))}|=\\
|\det \left(\Pi(X_t(f(x))) Df(X_t(x))\right) \circ  \left(\Pi(X_t(x)) DX_t(x)\right) \circ \left(\Pi(x) Df^{-1}(f(x))\right)|_{N(f(x))}|=\\
|\det \Pi(X_t(f(x))) Df(X_t(x))|_{N(X_t(x))}|.|\det \Pi(X_t(x)) DX_t(x)|_{N(x)}|.\\
. |\det \Pi(x) Df^{-1}(f(x))|_{N(f(x))}| = A_t.|\det P_t(x)|.C_t. 
\end{array}
\] 
Since $f$ is a $C^1$-diffeomorphism and its derivative preserves the vector field direction, there exists a constant $\tilde{K}>1$ such that 
\[
\tilde{K}^{-1} \leq \min\{A_t,C_t\} \leq \max \{A_t,C_t\} \leq \tilde{K}, \forall t\in \R.
\]
Therefore, for every $t\in \R$ we have that
\begin{equation}
\label{eq.provelemma}
|\log |\det P_t(x)| - \log |\det P_t(f(x))|| \leq 2\log \tilde{K}.
\end{equation}

Take $K > 2\log \tilde{K}$. If $f(x)$ did not belong to the orbit of $x$, by the UND property there would be $T\in \R$ such that 
\[
\displaystyle | \log \det P_T(x) - \log \det P_T(f(x))| > K. 
\]
This is a contradiction with (\ref{eq.provelemma}). Hence, $f$ fixes the orbit of $x$. 

Let $x\in \mathcal{D}^s$. Since $f$ commutes with the flow, it takes periodic orbits into periodic orbits of the same period. For a generic vector field, any two distinct periodic orbits have different periods. We conclude that $f$ fixes each periodic orbit. Since $f$ also takes stable manifolds into stable manifolds, we obtain that $f$ fixes stable manifolds. Therefore, we can apply the same calculations made before, restricting the jacobian to the stable manifolds and this will imply that $f$ fixes the orbit of $x$. The result then follows.
\end{proof}

This lemma states that the UND and $\mathrm{UND}^s$ properties provide some type of ``local'' triviality of the centralizer, meaning that $f$ fixes a dense set of orbits. Now we want to be able to extend this to the entire manifold. For that, we will use the LND property. We will need the following lemma to work with points in $M-\mathcal{CR}(X)$:

\begin{lemma}
\label{lem.opendensereparametrization}
There exists an open set $V\subset M - \mathcal{CR}(X)$, which is dense in $M - \mathcal{CR}(X)$, and a $C^1$-function $\tau:V\to \R$ such that $f(.) = X_{\tau(.)}(.)$ on $W$.
\end{lemma}

\begin{proof}
For each $x\in \mathcal{D}$, there is an unique number $T_x \in \R$ such that $f(x) = X_{T_x}(x)$. Let $p\in \mathcal{D}$. By Conley's theory, there is an open set $U$ such that $X_1(\overline{U}) \subset U$ and $p\in U- X_1(\overline{U})$ (see chapter 4 in \cite{AlongiNelson}). Since $f$ fixes the orbit of $p$, there is an unique $n\in \mathbb{Z}$ such that $f(p) \in X_{n-1}(U) - X_n(\overline{U})$.

Recall that for each $\delta>0$, we defined $\mathcal{N}(p,\delta) := \exp_p(N(p,\delta))$. Consider the map $\Psi(t,x) = X_t(x)$ defined on $(-2|T_p|,2|T_p|) \times \mathcal{N}(p,\delta)$. For $\delta$ small enough $\Psi$ is a $C^1$-diffeomorphism and $f(\mathcal{N}(p,\delta)) \subset X_{n-1}(U) - X_n(\overline{U})$. This implies that for each $q\in \mathcal{N}(p,\delta) \cap \mathcal{D}$, we have that $f(q) \subset \Psi((-2|T_p|,2|T_p|), q)$. Let $V_p= Im(\Psi)$.

Since $\mathcal{D}$ is dense, for each point $z\in \mathcal{N}(p,\delta)$, we can take a sequence $(z_n)_{n\in \N}$ contained in $\mathcal{N}(p,\delta) \cap \mathcal{D}$ and converging to $z$. Let $I_p = [-2|T_p|,2|T_p|]$, By the continuity of the flow, we have
\[
X_{I_p}(z_n) \xrightarrow{n\to + \infty} X_{I_p}(z).
\]
Since $f(z_n) \in X_{I_p}(z_n)$ and $\displaystyle \lim_{n\to +\infty} f(z_n)= f(z)$, we conclude that $f(z) \in X_{I_p}(z)$. In particular $f$ fixes the orbit of $z$. Therefore, for each $z\in \mathcal{N}(p,\delta)$, there is a number $T_z\in \R$ such that $f(z) = X_{T_z}(z)$. For each $x\in V_p$, consider the function $\hat{\tau}(x) = T_x$. Since $f$ and $\Psi^{-1}$ are $C^1$, we have that $\Psi^{-1} \circ f$ is also $C^1$. For any $z\in \mathcal{N}(p,\delta)$, we have $\Psi^{-1}(f(z))= (\hat{\tau}(z), z)$. From this formula, since $\hat{\tau}$ is constant along orbits and $\Psi^{-1} \circ f$ is $C^1$, we conclude that $\hat{\tau}$ is $C^1$ on $V_p$.

Take $V = \displaystyle \bigcup_{p\in \mathcal{D}} V_p$. Observe that $\hat{\tau}$ is uniquely defined on $V$ (since $f$ fixes orbits of $V$ and its points are non critical) and it is a $C^1$-function such that $f(.) = X_{\hat{\tau}(.)}(.)$ on $V$.  
\end{proof}

To deal with points in $int(\Omega(X))$ we need the following lemma:

\begin{lemma}
\label{lem.fixingorbitsstablemanifolds}
For each hyperbolic periodic point $p\in \mathrm{Per}(X)$, there exists a number $T_p\in \R$ such that for any $q\in W^s(p)$ it is verified that $f(q) = X_{T_p}(q)$.
\end{lemma}

\begin{proof}
Let $p\in \mathrm{Per}(X)$ be a hyperbolic periodic point of $X$. Since any two different periodic orbits have different periods (theorem \ref{thm.residualkupkabc}), $f$ fixes the orbit of $p$. Since the stable manifold $W^s(p)$ is a $C^1$-immersed submanifold, and $f$ fixes $W^s(p)$, we have that $f$ induces a $C^1$-diffeomorphim $f^s$ on $W^s(p)$, for the intrinsic topology. By lemma \ref{lem.fixesDDs}, $f$ fixes the orbits of the points in $\mathcal{D}^s_p$.

The points in a stable manifold are non-recurrent for the intrinsic topology, hence, by an argument similar to the one in the proof of lemma \ref{lem.opendensereparametrization}, we obtain an open set $V^s_p$ which is dense in $W^s(p)$, and a $C^1$-function $\tau^s: V^s_p \to \R$ such that $f^s(.) = X_{\tau^s(.)}(.)$. 

We claim that $\tau^s(.)$ is constant on $V_p^s$. The proof of this fact is essentially contained in the proof of claim \ref{claim.constantalongstablemanifoldperiodic} in lemma \ref{lemma.functionperiodicpoints}. We sketch the proof of this fact here, for more details on the arguments see the proof of claim \ref{claim.constantalongstablemanifoldperiodic} in lemma \ref{lemma.functionnoncriticalpoints}. 

First we prove that $\tau^s$ is constant on each connected component of $V^s_p$. If it was not the case, we could find two points $x,y\in W^{ss}(p) -\{p\}$ such that $0<|\tau^s(x) - \tau^s(y)|<\pi(p)$, where $\pi(p)$ is the period of $p$. This implies that $f(x) \in X_{\tau^s(x)}(W^{ss}(p))$ and $f(y) \in X_{\tau^s(y)}(W^{ss}(p))$. From this, one can deduce that $W^{ss}(f(x)) \cap W^{ss}(f(y)) = \emptyset$, which is a contradiction with the fact that $f(W^{ss}(p)) = W^{ss}(f(p))$. This implies that $\tau^s$ is constant in each connected component of $V^s_p$. 

Recall that on $TW^s(p)$ we defined the normal bundle $N_s = \langle X \rangle^{\perp}$. Let $\Pi^s: TW^s(p) \to N_s$ be the orthogonal projection on $N_s$ and let $P^s_t(.)$ be the linear Poincar\'e flow restricted to $W^s(p)$.

Suppose that $\tau^s$ does not take the same value in every connected component of $V^s_p$. Let $V_1$ and $V_2$ be two connected components of $V^s_p$ such that the numbers $T_1 := \tau^s|_{V_1}$ and $T_2:= \tau^s|_{V_2}$ are not equal. We may choose $q_1^n$ a sequence in $V_1$ converging to $p$ and $q^n_2$ a sequence in $V_2$ converging to $p$. Since $f^s$ is $C^1$, we would have that
\[
\begin{array}{rcl}
\displaystyle \lim_{n\to + \infty} \Pi^s(f^s(q^n_1))Df^s(q^n_1)\Pi^s(q^n_1) & = &\\
\displaystyle \lim_{n\to + \infty} \Pi^s(f^s(q^n_2))Df^s(q^n_2)\Pi^s(q^n_2) &= &\Pi^s(f^s(p))Df^s(p)\Pi^s(p).
\end{array}
\]
However, by the hyperbolicity of $p$ and since $T_2 \neq T_1$, we have $\|P^s_{T_1}(p)\| \neq \|P^s_{T_2}(p)\|$. This is a contradiction and $\tau^s$ is constant on $V^s_p$. Hence, there exists $T_p\in \R$ such that $\tau^s(q) = T_p$ for every $q\in V^s_p$. This easily implies that $f(q) = X_{T_p}(q)$, for any $q\in W^s(p)$.  
\end{proof}

By theorem \ref{thm.residualkupkabc}, in each connected component $O$ of $int(\Omega(X))$, there exists a periodic point $p$ whose stable manifold is dense in $O$. From lemma \ref{lem.fixingorbitsstablemanifolds}, there exists a number $T_p$ such that $f(q) = X_{T_p}(q)$, for any $q\in W^s(p)$. This implies that $f(q) = X_{T_p}(q)$ for any $q\in O$. 

Consider the open and dense set $W = V \cup int(\Omega(X))$. From the discussion above, there is a $C^1$-function $\hat{\tau}:W \to \R$ such that $f(.) = X_{\hat{\tau}(.)}(.)$ on $W$.   
  
\begin{lemma}
\label{lem.lndisgood}
There exists a constant $T>0$ such that $|\hat{tau}(x)| \leq T$ for any $x\in W$.
\end{lemma}

\begin{proof}
For $x\in W$ and any vector $v\in T_xM$, the following formula holds:
\begin{equation}
\label{eq.formuladerivative}
Df(x)v = DX_{\hat{\tau}(x)}(x)v + X(X_{\hat{\tau}(x)}(x))D\hat{\tau}(x)v
\end{equation}
From this formula, one can see that $\Pi(f(x)) \circ Df(x)|_{N(x)} = P_{\hat{\tau}(x)}(x)$. Take $K> \|f\|_{C^1}$ and let $T=T(K)>0$ be the uniform time given by the LND property. 

If $\hat{\tau}$ was not uniformly bounded, there would be a point $x\in W$ such that $|\hat{\tau}(x)| > T$. By the LND property, there exists a point $y\in orb(x)$ such that
\[
\max\{\|P_{\hat{\tau}(x)}(y)\|, \|P_{-\hat{\tau}(x)} (X_{\hat{\tau}(x)}(y))\|\}> K.
\]

Then,
\[
\max\{ \|Df(y)\|, \|Df^{-1}(f(y))\|\} \geq \max\{\|P_{\hat{\tau}(x)}(y)\|, \|P_{-\hat{\tau}(x)} (X_{\hat{\tau}(x)}(y))\|\}> K.
\]  
This is a contradiction since $K> \|f\|_{C^1}$. Therefore, $\hat{\tau}$ is uniformly bounded on $W$.
\end{proof}

Let $I= [-T,T]$. From lemma \ref{lem.lndisgood}, we obtain that for any point $x\in W$, it is verified that $f(x) \in X_I(x)$. Since $W$ is dense, for any point $z\in M$, there is a sequence $(y_n)_{n\in \N}$ contained in $W$ such that $\displaystyle \lim_{n\to +\infty} y_n = y$. Since $X_I(y_n) \xrightarrow{n\to +\infty} X_I(y)$, and by the continuity of $f$, we conclude that $f(y) \in X_I(y)$. In particular, $f$ fixes every orbit of $X$ and $f(p) \in X_{[-T,T]}(p)$, for any $p\in M$. 

By proposition \ref{prop.mainpropextension}, there is a continuous function $\tau:M \to \R$ such that $f(.) = X_{\tau(.)}(.)$. This proves that the centralizer is quasi-trivial. 

In the case that $X$ has at most countably many chain-recurrent classes, since the function $\tau$ is an $X$-invariant continuous function, using the same arguments as in section $6.4$ of \cite{LeguilObataSantiago}, we conclude that $\tau$ is a constant function. In particular, the centralizer is trivial. 

\section{A remark on the centralizers}
\label{sec.furtherremarks}
As we mentioned before, the diffeomorphism centralizer is less rigid than the vector field centralizer, mentioned in the introduction. In this section we give one example that justifies it.

By the work of \cite{KatoMorimoto} the vector field centralizer of an Anosov flow is trivial. However, it is easy to construct an Anosov flow and a diffeomorphism that commutes with the flow and which does not fixes the orbits of the flow. For example, one can take two hyperbolic matrices $A,B \in SL(n,\mathbb{Z})$ that commute and such that the group they generate is isomorphic to $\mathbb{Z}^2$; that is, $A^l B^k = Id$ if and only if $l=k=0$. For example, in dimension $3$ one may consider the matrices
\[
A=
\begin{pmatrix}
3 & 2 & 1\\
2 & 2 & 1\\
1&1&1
\end{pmatrix}
\textrm{ and }
B = 
\begin{pmatrix}
2& 1 & 1\\
1& 2 & 0\\
1 & 0 & 1
\end{pmatrix}
.
\]

These matrices induce Anosov diffeomorphisms on the torus $\mathbb{T}^3$. Consider the suspension flow of $A$, this gives an Anosov flow $X^A_t$. Using $B$ one can easily construct a diffeomorphism $f_B$ that commutes with the the flow $X^A_t$ and that does not fixes orbit. In particular, the centralizer of $X^A_t$ is not trivial.

\bibliographystyle{alpha}

\information

\end{document}